\documentclass{svjour2}
\smartqed  
\usepackage{graphicx}
\usepackage[utf8x]{inputenc}
\usepackage{amssymb,latexsym}
\usepackage{amsmath}
\usepackage{ulem}

\def\tref#1{Theorem~$\ref{#1}$}
\def\Tref#1{Table~$\ref{#1}$}
\def\pref#1{Proposition~$\ref{#1}$}
\def\fref#1{Figure~$\ref{#1}$}\def\cref#1{Corollary~$\ref{#1}$}
\def\sref#1{Section~$\ref{#1}$}
\def\lref#1{Lemma~$\ref{#1}$}
\usepackage{colortbl}
\providecommand{\blk}{\cellcolor[gray]{.00}}

\begin{document}

\title{A cube dismantling problem 
related to bootstrap percolation\thanks{Research is supported by the Hungarian National Science 
Foundation (OTKA) Grant K~76099,
and Australian Research Council grant
DP120100197.}}

\author{J\'anos Bar\'at      \and
         Ian M. Wanless
}

\institute{J\'anos Bar\'at   \at
           \and
           Ian M. Wanless \at
              School of Mathematical
Sciences, Monash University\\ 
Victoria 3800 Australia
}

\date{}

\maketitle

\begin{abstract}

An $n\times n\times\dots\times n$ hypercube is made from $n^d$ unit
hypercubes.  Two unit hypercubes are neighbours if they share a
$(d-1)$-dimensional face.  In each step of a dismantling process, we
remove a unit hypercube that has precisely $d$ neighbours.  A move is
{\it balanced} if the neighbours are in $d$ orthogonal directions.  In
the extremal case, there are $n^{d-1}$ independent unit hypercubes
left at the end of the dismantling.  We call this set of hypercubes a
{\it solution}.  If a solution is projected in $d$ orthogonal
directions and we get the entire $[n]^{d-1}$ hypercube in each
direction, then the solution is {\it perfect}.

We show that it is possible to use a greedy algorithm to test whether
a set of hypercubes forms a solution.  Perfect solutions turn out to be
precisely those which can be reached using only balanced moves. Every
perfect solution corresponds naturally to a Latin hypercube. However,
we show that almost all Latin hypercubes do not correspond to
solutions.

In three dimensions, we find at least $n$ perfect solutions for every
$n$, and we use our greedy algorithm to count the perfect solutions
for $n\le6$.  We also construct an infinite family of imperfect solutions
and show that the total size of its three orthogonal projections is
asymptotic to the minimum possible value.

Our results solve several conjectures posed in a proceedings paper by
Bar\'at, Korondi and Varga.

If our dismantling process is reversed we get a build-up process very
closely related to well-studied models of bootstrap percolation.  We
show that in an important special case our build-up reaches the same
maximal position as bootstrap percolation.

\end{abstract}

\section{Introduction}
\label{intro}

A $d$-dimensional hypercube of edge-length $n$ is composed of $n^d$
unit hypercubes, which we call ``cubes'' for short.  We identify these
cubes with $d$-tuples of integers from $\{1,2,\dots,n\}$, and think of
them as vertices of a graph $[n]^d$.  Two vertices are {\it
  neighbours} if the corresponding two cubes have a
$(d-1)$-dimensional face in common.  We will consider sequences of
induced subgraphs of $[n]^d$.  Starting from $[n]^d$, we successively
remove vertices of degree $d$ from the current graph until no further
vertices can be removed.  Here the degree of a vertex is a dynamic
notion referring to the degree in the current graph.  We prove a
number of properties of the positions that may be reached by such a
{\it dismantling} process. In doing so, we settle several conjectures
posed in~\cite{hj2011}.

We also study the reverse of our dismantling process, which we call a
{\it build-up}. The build-up process is similar to well-studied models
of bootstrap percolation, see \cite{ami,random,holroyd,min,sau} for
example. One difference with our process is that a new site is added
only when the number of neighbours exactly matches a prescribed value,
whereas in bootstrap percolation ``exactly matches'' would be replaced
by ``matches or exceeds''.  Therefore, extremal results in bootstrap
percolation theory can serve as upper bounds for our problem.  Another
issue is that we only add one vertex at a time, while in bootstrap
percolation all sites satisfying the threshold condition are added
simultaneously.  In our setting, two vertices that correspond to
independent vertices of $[n]^d$ can be added
simultaneously. However, if the vertices are neighbours and both have
degree $d$, then they cannot both be added.  This changes the nature
of the problem.  For bootstrap percolation it is of fundamental
importance to establish bounds on the size of minimal percolating sets
\cite{min,Riedl} or on the percolating time \cite{time}.  In contrast,
the objects corresponding to percolating sets in our setting all have
the same size (see \pref{p:opti} below). Therefore, since
we only add one cube at a time, they all take the same time to build-up.
Despite these differences, we show in \pref{p:uniqmax} that in an
important special case, our build-up process is effectively identical
to bootstrap percolation. Another similarity is that the so-called
modified bootstrap percolation in $[n]^d$ requires at least one
neighbour in each of the $d$ directions \cite{holroyd}.  This property
is analogous to the idea of a balanced move in our investigation.

The Perimeter Lemma (see e.g. \cite{random}) is a well-known tool
in bootstrap percolation. The same idea will help to prove a number of
extremal results in our paper, beginning with this:

\begin{proposition}\label{p:opti}
Starting from $[n]^d$, there are at least $n^{d-1}$ cubes left
after any sequence of moves, where each move removes a vertex of
degree~$d$. 
\end{proposition}
 
\begin{proof}
A face of a cube is {\it invisible} if there is another cube attached to it.
Otherwise the face is {\it visible}. The {\it visible surface} is the total number of visible faces.
In the beginning, the surface of the $d$-dimensional hypercube is
visible.  Therefore, the total visible surface of the unit cubes is
$2dn^{d-1}$.  In each move the degree condition guarantees that the
visible surface is unchanged.  Therefore, at the end, the visible
surface is $2dn^{d-1}$.  That requires at least $n^{d-1}$ unit cubes. \qed
\end{proof}

A {\it position} refers to any subset of cubes in $[n]^d$, that is,
an induced subgraph of $[n]^d$. 
In the extremal case of \pref{p:opti} the $n^{d-1}$ unit cubes correspond to
independent vertices (no two of them are neighbours) of $[n]^d$.  
In our terminology $n^{d-1}$ independent vertices of $[n]^d$ form a {\it base position}.  
We colour the vertices in a position {\it black},
and the remaining vertices of $[n]^d$ {\it white}.  
The degree of a white vertex is its number of black neighbours.  
One question is the following.
For a given base position, can we successively remove all white
vertices according to the degree $d$ rule?  If the answer is yes, then
the base position is a {\it solution}.  A sequence of moves starting
from $[n]^d$ leading to a solution is a {\it dismantling}.  The
reverse process is a {\it build-up}.

Each vertex of $[n]^d$ is identified with a $d$-tuple
$(x_1,\dots,x_d)$.  A {\it line} is the set of $n$ vertices found by
fixing $(d-1)$ coordinates and allowing the other coordinate to vary. A
{\it section} is the set of $n^{d-1}$ vertices found by fixing one
coordinate and allowing the other $(d-1)$ coordinates to vary.  A
section is {\it facial} if its fixed coordinate has value $1$ or $n$.
In \sref{s:geom} we consider some properties of the intersection of a
solution and a section.  A {\it projection} of a subset of the vertices
in $[n]^d$ is a projection to $[n]^{d-1}$ obtained by omitting any one
specific coordinate.  The projection is {\it full} if it is 
surjective onto $[n]^{d-1}$.

Each step in a dismantling, that is, the removal of a vertex of degree
$d$, is a {\it move}.  A move is {\it balanced} if the $d$ neighbours lie
in different lines.  A dismantling or build-up is {\it balanced} if
all of its moves are balanced, otherwise it is unbalanced.  We show in
Section~\ref{gandb} that if a dismantling to a solution is balanced,
then the solution corresponds to a Latin hypercube.

Any base position with full
projections in all $d$ orthogonal directions is {\it perfect}.  A
position is {\it convex} if no line of $[n]^d$ contains one or more
white cubes between black cubes.  In particular, a perfect base
position is convex.  A position is {\it minimal} if there is no black
vertex with precisely $d$ neighbours.  That is, no more dismantling
moves are possible.  A position is {\it maximal} if there is no white
vertex with precisely $d$ neighbours.  That is, no more build-up moves
are possible.  We are mainly interested in the maximal configuration
$[n]^d$, and the minimal configurations that are solutions.  In
Theorem~\ref{t:notcontain} we show that the symmetric difference of
any two minimal or any two maximal positions is non-empty.

A dismantling gives us an edge-decomposition of $[n]^d$ into copies
of the complete bipartite graph $K_{1,d}$, also known as the $d$-star.
Decompositions are well-studied in graph theory and design theory
\cite{bar,bosak,dor,hof,w3f}.  The existence of a star decomposition does
not imply a dismantling since dismantling might plausibly reach a
subgraph that has a decomposition but no vertex of degree $d$.
Hoffmann~\cite{hof} gave a necessary and sufficient condition for star
decompositions.  His third condition is a Hall-type criterion, which
is difficult to check in general.  It is known \cite{dor} that to
decide if there exists a $d$-star decomposition of a given graph $G$
is $NP$-complete for any $d\ge 3$.  Recently, Thomassen~\cite{w3f}
proved that any $8$-edge-connected graph $G$ has a $3$-star
decompositions.  Notice that $[n]^3$ is only $3$-edge-connected and
has maximum degree $6$.  Therefore, the $3$-star decomposition result
implied by our \tref{cyc} is non-trivial. See also \cite{bar}, where
$K_{1,3}$-decompositions are related to orientations and flows.

\section{Greedy and balanced}\label{gandb}

If $B$ is an imperfect solution, then any dismantling to $B$ is
unbalanced as was observed in \cite{hj2011} for the $3$-dimensional case.  Formally, we have the
following

\begin{lemma}\label{l:balmov}
If all moves are balanced in a dismantling to a solution $B$,
then $B$ is perfect.
\end{lemma}

\begin{proof}
Assume to the contrary, that two black vertices $u$ and $v$ have the
same projection in some direction $r$, and there are only white
vertices between $u$ and $v$.  Say $u$ and $v$ are at distance $k$.
Every edge is removed during
dismantling. However, any balanced move removes one white vertex and
one edge in direction $r$, so it is impossible to remove the $k$ edges
between $u$ and $v$ with the $k-1$ available white vertices.  This
contradiction proves that the projection is full in all
directions.  
\qed \end{proof}

It is challenging to compare the importance of balanced and unbalanced
moves.  Note that every dismantling begins with a balanced move.
There are dismantlings in $3$ dimensions with balanced moves only, we
show one particular solution in Section~\ref{s:cyclic}.  On the other
hand, there seem to be many more imperfect solutions than perfect
ones.  Theoretically, it is plausible to replace a few balanced moves
by a few imbalanced moves.  However, the authors of \cite{hj2011}
conjectured that the converse of \lref{l:balmov} holds.  We settle that
conjecture in our next Proposition, but first we need a preliminary
result:

\begin{lemma}\label{convex} 
Let $C$ be a convex position.\\
$(0)$ Any build-up from $C$ is balanced, and produces only convex positions.\\
$(1)$ Any dismantling to $C$ is balanced.\\
$(2)$ Only convex positions can be reached by balanced dismantling from $C$.
\end{lemma}

\begin{proof}
First consider a build-up from $C$.
Any unbalanced move requires two non-consecutive cubes in the same line.
Therefore, only balanced moves are possible in a convex position.

Assume $D$ is a non-convex position formed by 
adding a single vertex $v$ to the convex position $C$.  Since $D$ is
non-convex, there are two non-consecutive vertices in a line $l$ such
that all vertices between them are missing.  Clearly one of the
vertices must be $v$, the other one is, say, $u$.  Vertex $v$ was
added in a balanced move, therefore there is a neighbour $x$ of $v$ in
line $l$.  There are no vertices between $v$ and $u$ in $D$, therefore
$x$ and $u$ make $C$ non-convex, a contradiction. So, by induction, only
convex positions can be reached by building-up from $C$.

This proves (0), and (1) follows immediately since
dismantling is the reverse of build-up. Part (2) is straightforward,
since each move of a dismantling preserves convexity unless we remove
a cube between two other cubes, and such a move is not balanced.  
\qed \end{proof}

A $d$-dimensional {\it Latin hypercube} of order $n$ is a
$d$-dimensional array of $n$ symbols in which each symbol occurs
exactly once in each line \cite{lhc}.  Each position $P$ in a
dismantling can be associated with a $(d-1)$-dimensional array
$M_P=[m_{x_1x_2\dots x_{d-1}}]$ of sets such that $z\in m_{x_1x_2\dots
  x_{d-1}}$ if and only if the cell $(x_1,\dots,x_{d-1},z)$ is included in
$P$.  It is plausible that $M_P$ might be a Latin hypercube with the
$n$ symbols being the $n$ possible singleton sets. Positions that
correspond to Latin hypercubes in this way turn out to be quite
special:

\begin{proposition}\label{p:convperfLS}
For a solution $B$ the following statements are equivalent:\\
$(0)$ $B$ is convex.\\
$(1)$ $B$ is perfect.\\
$(2)$ $M_B$ is a Latin hypercube.\\
$(3)$ There is a balanced dismantling of $[n]^d$ to $B$.\\
$(4)$ Every dismantling of $[n]^d$ to $B$ is balanced.\\
$(5)$ There is a balanced build-up from $B$ to $[n]^d$.\\
$(6)$ Every build-up from $B$ to $[n]^d$ is balanced. 
\end{proposition}

\begin{proof}
If $B$ is convex, then there cannot be more than one black vertex in any line, 
since $B$ consists of independent vertices. However, there has
to be one vertex per line on average, which means there must be
exactly one vertex in each line.  So $B$ is perfect.  Conversely, if
$B$ is perfect then there is exactly one black vertex in each line,
which means that $B$ is necessarily convex.  We conclude that (0) and
(1) are equivalent.

The cells of $M_B$ consist of singletons if and only if there is
exactly one vertex of $B$ in each line in the $d$-th direction. This
is the same as specifying one orthogonal projection of $B$ to be
full. The other $(d-1)$ orthogonal projections of $B$ are full if and
only if the singletons in each line of $B$ are distinct.  So (1) is
equivalent to (2) by the definition of a Latin hypercube.

Dismantling and build-up are the reverse of each other, and balanced
moves mean the same thing in both processes. Hence (3) is equivalent
to (5) and (4) is equivalent to (6). We are assuming that $B$ is a
solution so there is a dismantling from $[n]^d$ to $B$. Thus (4) implies
(3). \lref{l:balmov} shows that (3) implies (1), and 
\lref{convex} shows that (0) implies (6). The proposition follows.
\qed \end{proof}

We stress that \pref{p:convperfLS} applies only to
solutions $B$, not to all base positions $B$. We will see in
\sref{search} that (2) is very far from implying (3) for a general
base position.

\medskip

When the authors of \cite{hj2011} tried checking if base positions are
solutions by exhaustive search, most of the search time was spent
backtracking in the search tree.  They conjectured that this is
unnecessary in some sense. We validate this conjecture here.  The
steps of a build-up can be encoded as follows: in each step, we record
the vertex $v$ that we are adding, and store its set $N(v)$ of neighbours for
validation of the degree.  We use $v^*$ as short-hand for the pair
$(v,N(v))$.

\begin{theorem}\label{t:notcontain}
Let $P$ be any position. There cannot be two distinct maximal positions
$M_1\subset M_2$ such that both $M_1$ and $M_2$ can be 
reached by build-up from $P$.
Similarly there cannot be two distinct minimal positions
$M_1\subset M_2$ such that both $M_1$ and $M_2$ can be 
reached by dismantling from $P$.
\end{theorem}

\begin{proof}
Assume to the contrary that some build-up $u_1^*,u_2^*,\dots,u_f^*$
from $P$ stops in a maximal position $M_1$,
while another sequence $v_1^*,v_2^*,\dots,v_{g}^*$ builds up from $P$ to
$M_2$, where $M_1\subsetneq M_2$. 

First suppose that
$\{u_1^*,u_2^*,\dots,u_f^*\}\subset\{v_1^*,v_2^*,\dots,v_{g}^*\}$.
Since $M_1\neq M_2$, there is at least one vertex in the sequence
$v_1,v_2,\dots,v_g$ that does not occur in $\{u_1,u_2,\dots,u_f\}$.  Let
$v_j$ be the first such vertex in the sequence, and suppose
$v_j^*=(v_j,\{x_1,\dots,x_d\})$.  We claim that $v_j^*$ is a possible move in
position $M_1$, contradicting the maximality of $M_1$.  By choice of
$v_j$, each of $x_1,\dots,x_d$ must be in $P$ or $\{u_1,\dots,u_f\}$ and hence
each is in $M_1$.  Thus $v_j$ has at least $d$ neighbours in $M_1$.
Suppose $v_j$ has another neighbour $w$ in $M_1$. Since 
$w\notin N(v_j)$ it cannot be that $w$ is in $P$, so there must
be some $k$ for which $u_k=w$. By assumption, $u_k^*=v_i^*$ for some $i$.
If $i<j$, then $w=u_k=v_i\in N(v_j)$ which contradicts the choice of $w$.
If $i>j$, then $v_j\in N(v_i)=N(u_k)\subset M_1$ which contradicts the choice
of $v_j$. It follows that $v_j$, which
is not present in $M_1$, has exactly $d$ neighbours in $M_1$.
Thus $v_j^*$ is a possible move in position $M_1$, as claimed.

It remains to consider the possibility that $u_i=v_k$ but $u_i^*\ne v_k^*$ for some $i,k$. 
We choose the first such $u_i$
in the sequence $u_1,\dots,u_f$. Since $u_i^*\ne v_k^*$, there is 
$x\in N(u_i)\setminus N(v_k)$ and this implies that $x=u_j$ for some
$j<i$. By choice of $i$, this means that $u_j^*=v_l^*$ for some $l$.
As $j<i$ we have $v_k=u_i\notin N(u_j)=N(v_l)$, which implies $l<k$.
Hence $x=u_j=v_l\in N(v_k)$, contradicting the choice of $x$
and proving the first claim of the theorem.

The proof of the other claim is similar. Suppose that there are two
minimal positions $M_1\subsetneq M_2$ such that $M_2$ can be
reached from $P$ by the dismantling $u_1^*,u_2^*,\dots,u_f^*$,
while another sequence $v_1^*,v_2^*,\dots,v_{g}^*$ dismantles from $P$ to
$M_1$. If $\{u_1^*,u_2^*,\dots,u_f^*\}\subset\{v_1^*,v_2^*,\dots,v_{g}^*\}$
then the first move in $v_1^*,v_2^*,\dots,v_{g}^*$ that is not in
$u_1^*,u_2^*,\dots,u_f^*$ will be a valid dismantling move from $M_2$,
contradicting the minimality of $M_2$. Otherwise, we have $u_i=v_k$ 
but $u_i^*\ne v_k^*$ for some $i,k$. Taking the first such $i$ yields
a similar contradiction to the build-up case.
\qed \end{proof}

Since any position that can be reached from $P$ is a subset of $[n]^d$, 
we have:

\begin{corollary}\label{c:strong}
Let $P$ be any position that can be reached from $[n]^d$ by dismantling.
The only maximal position that can be reached by building up from $P$
is $[n]^d$.
\end{corollary}

Crucially, this last result entitles us to use the following greedy
algorithm to check candidate solutions.  We repeatedly traverse the
white vertices (in any order) and add any of degree $d$, until a maximal
position is reached.

\begin{corollary}\label{c:LSgreedy}
A base position is a solution if and only if the greedy algorithm
terminates with $[n]^d$.
\end{corollary}

In \cref{c:LSgreedy}, if we end up with a maximal position other than $[n]^d$,
then we started with a position that was not a
solution.  The maximal position is far from being determined by the
initial position.  We found a set of 25 independent vertices in $[5]^3$
from which different build-ups reach maximal positions of any size in
$\{37,38,39,40,43,46,56,57,58,59,60,61,63\}$.  We also found a set of
35 vertices in $[5]^3$ such that various build-ups resulted in maximal
positions of each size in $\{93,94,\dots,119\}$.  Figures~\ref{93} and
\ref{119} show the smallest and largest maximal positions for this
example. The numbers in shaded squares show the order in which the cubes are
added. Bold numbers in white squares are the degrees of white vertices in the final
position.  The absence of a bold $3$ certifies that no more moves are
possible.

\begin{figure}[ht]
  \begin{center}
    \includegraphics[width=0.85\textwidth]{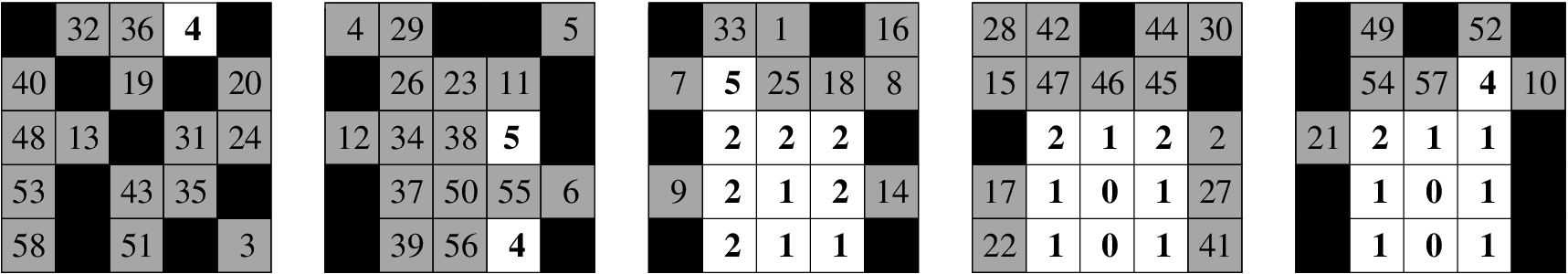}
    \caption{\label{93}The smallest maximal position in the five layers of $[5]^3$.}
  \end{center}
\end{figure}

\begin{figure}[ht]
  \begin{center}
    \includegraphics[width=0.85\textwidth]{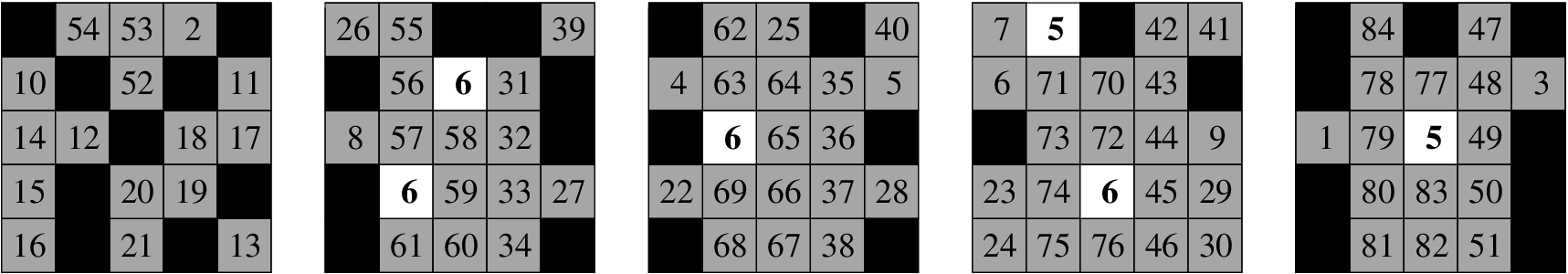}
    \caption{\label{119}The largest maximal position in the five layers of $[5]^3$.}
  \end{center}
\end{figure}
These examples contrast sharply with
\cref{c:strong}.  In \pref{p:uniqmax} below, we will see another
situation where only one maximal position can be reached.

Next we prove another conjecture from \cite{hj2011}.
It states that solutions are not only stable, but in some sense the
build-up is unique.

\begin{theorem} \label{anytwo}
Let $P_1$ and $P_2$ be any two positions, and let 
$u_1^*, u_2^*,\dots, u_{f}^*$ and $v_1^*, v_2^*,\dots, v_{f}^*$,
be any two build-ups from $P_1$ to $P_2$. There exists some
permutation $\sigma$ of $\{1,2,\dots,f\}$ such that $u_{\sigma(i)}^*=v_i^*$
for $1\le i\le f$.
\end{theorem}

\begin{proof}
Assume to the contrary, that there exists a $j$ for which 
$u_j^*\neq v_i^*$ for $1\le i\le f$.
Let us fix $j$ to be the smallest such index. 
Since both sequences reach the same position, $u_j=v_k$ for some $k$. 

Suppose $u_j^*=(u_j,\{x_1,\dots,x_d\})$. By assumption, $u_j^*\neq v_k^*$ so
at least one of $x_1,\dots,x_d$ must occur after $v_k$ in the sequence $v_1,\dots,v_f$.
Without loss of generality, suppose that $x_d$ does. Now $x_d$ is not in $P_1$,
but is in $N(u_j)$, so $x_d=u_h$ for some $h<j$. By choice of $j$, it follows
that $u_h^*=v_e^*$ for some $e$. Note that $e>k$ by choice of $x_d$.
Now $v_k=u_j\notin N(u_h)=N(v_e)$ since $h<j$, but this
contradicts the fact that $v_k\in N(v_e)$ as $k<e$.
\qed \end{proof}

Of course, a similar statement holds for any two dismantlings from 
$P_2$ to $P_1$, since a dismantling sequence is just a build-up sequence
in reverse.

In bootstrap percolation, sites in $[n]^d$ are initially declared
active independently with probability $p$.  At subsequent steps, an
inactive site becomes active if it has at least $d$ active neighbours.
In the modified bootstrap percolation model, sites in $[n]^d$ are
initially declared active independently with probability $p$. At
subsequent steps, an inactive site becomes active if it has at least
one active neighbour in each of the $d$ directions. An active site
remains active forever.  In the introduction we mentioned several
differences between our build-up process and bootstrap
percolation. However, there is an important case in which they are
essentially the same.

\begin{proposition}\label{p:uniqmax}
  If $C$ is any convex position, then the same maximal position will
  ultimately be reached from $C$ by any sequence of moves
  regardless of whether they are dictated by our build-up process, by
  bootstrap percolation, or by modified bootstrap percolation.
\end{proposition}

\begin{proof}
  Let $S$ be the set of white vertices that have at least $d$ black
  neighbours in $C$.  Since $C$ is convex, no white vertex can have
  more than one black neighbour in any line of $C$. Hence every vertex
  in $S$ has exactly $d$ black neighbours, and they are in $d$
  orthogonal directions. Thus $S$ is the set of sites that would be
  filled by a single step of either bootstrap percolation or modified
  bootstrap percolation. Moreover, the vertices in $S$ are
  independent; for suppose that two of them, $u$ and $v$ were
  neighbours.  Since they are in $S$, they both have a black
  neighbour in the line that contains $u$ and $v$. But these black
  neighbours would necessarily have the white vertices $u$ and $v$
  between them, which breaches the assumption that $C$ is convex. So
  $S$ is an independent set and all vertices in it can be added (in
  any order, and using only balanced moves) by our build-up
  process. By \lref{convex}, $C\cup S$ is convex, so we can apply the
  same argument again.

  Let $D$ be the maximal position reached from $C$ by bootstrap
  percolation (or modified bootstrap percolation). By induction using
  the above argument we know that $D$ can be reached by build-up from
  $C$. So it only remains to show that no other maximal position $D'$
  can be reached by build-up from $C$. Suppose $w$ is the first vertex
  that was added in the build-up to $D$ but is not present in
  $D'$. Now $w$ has $d$ neighbours in $D'$, namely the neighbours
  that were used to add $w$ in the build-up to $D$. Also $w$ cannot
  have more than $d$ neighbours in $D'$, since $D'$ is convex by
  \lref{convex}. So $w$ is available as a move in $D'$, contradicting
  the maximality of $D'$. We conclude that $w$ does not exist, meaning
  that $D$ is a subset of $D'$. By a symmetric argument $D'$ is a
  subset of $D$, so they are in fact equal. In other words, from any
  convex position there is a unique maximal position that can be
  reached by building-up.
\qed \end{proof}

\section{Additional geometric properties}\label{s:geom}

In this section we consider further properties of solutions, particularly
their sections and projections. For this purpose, it is natural to assume
$d>2$.

We start by considering the number of black vertices in a section.
Certainly, a section cannot consist of white vertices only, since then
the last vertex removed from the section in a hypothetical dismantling
would have degree at most 2, a contradiction.  So each section
contains at least one black vertex.  For facial sections we can say
something much stronger using the Perimeter Lemma.

\begin{lemma}\label{facial}
For any solution every facial section of $[n]^d$ contains at least $n^{d-2}$ black vertices.
\end{lemma}

\begin{proof} 
Assume there is a given solution, and the number of black vertices is
$x$ in a facial section $F$.  During the dismantling, we have to
remove the white vertices of $F$.  Let us enumerate the available
edges for this process.  There are $(d-1)(n^{d-1}-n^{d-2})$ edges within $F$, and
there are $n^{d-1}-x$ edges between the white vertices of $F$ and the
section adjacent to $F$.  In each step, the removal of a white vertex
uses $d$ edges.  Therefore, 
$(d-1)(n^{d-1}-n^{d-2})+n^{d-1}-x\ge d(n^{d-1}-x)$, which implies
$x\ge n^{d-2}$.  \qed
\end{proof}

The above result is a subcase of condition (iii) in Hoffman's
paper~\cite{hof}.  An analogous argument applied to a non-facial
section gives no extra information.  On the other hand, we get a
stronger result assuming that the dismantling is balanced.  By
\pref{p:convperfLS} the solution is then perfect and convex, and it
follows that each section contains precisely $n^{d-2}$ black vertices.
We note that the converse does not hold. \fref{f:npersect} shows an
example for $n=4$ and $d=3$ where each section contains precisely
$n^{d-2}=4$ black vertices but the solution is not perfect.

\begin{figure}[h]
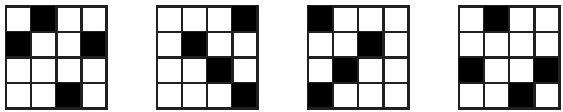

\[
\begin{array}{|p{6pt}|p{6pt}|p{6pt}|p{6pt}|}
\hline
&\blk&&\\\hline
\blk&&&\blk\\\hline
&&&\\\hline
&&\blk&\\\hline
\end{array}
\qquad
\begin{array}{|p{6pt}|p{6pt}|p{6pt}|p{6pt}|}
\hline
&&&\blk\\\hline
&\blk&&\\\hline
&&\blk&\\\hline
&&&\blk\\\hline
\end{array}
\qquad
\begin{array}{|p{6pt}|p{6pt}|p{6pt}|p{6pt}|}
\hline
\blk&&&\\\hline
&&\blk&\\\hline
&\blk&&\\\hline
\blk&&&\\\hline
\end{array}
\qquad
\begin{array}{|p{6pt}|p{6pt}|p{6pt}|p{6pt}|}
\hline
&\blk&&\\\hline
&&&\\\hline
\blk&&&\blk\\\hline
&&\blk&\\\hline
\end{array}
\]
\caption{\label{f:npersect}The four planar sections of an imperfect 
solution in $[4]^3$ with $4$ black vertices per section.}
\end{figure}

In applications such as those mentioned in \cite{hj2011}, there might
be a condition on the location of the black cubes.  One strong
condition is to restrict the black cubes to being in $d$ sections.
For $d=3$ and $n=2^t$ for a positive integer $t$, we now show a
solution satisfying this extra condition.  The construction relies on
the following observation.  If there are three mutually orthogonal
sections in which all cubes are present, then the rest of $[n]^3$ can
be built-up.  What is left to do is the selection of a base position
in this restricted space of three sections.

\begin{lemma}[The corridor idea]
Let $n=2^t$ for $t\ge 1$. There exists a solution $I_n$ in
which the black cubes are contained in three orthogonal facial
sections of $[n]^3$.
\end{lemma}

\begin{proof}
We use induction on $t$.  For $t=1$, there is only one solution.  For
any other $t$, let $I_{2^t}$ contain the vertex $(1, 1, 1)$. Let the
other black cubes be contained in squares of size $(2^t-1)\times
(2^t-1)$ in the facial sections $(1,y,z)$, $(x,1,z)$, $(x,y,1)$, where
$x,y,z\ge 2$.  The positions of the black cubes in these sections are
defined recursively.  Consider four copies of the example of size
$(2^{t-1}-1)\times (2^{t-1}-1)$ in the corners plus an extra black
cube in the centre, see \fref{pepita}.

We can build-up $[n]^3$ from this base position as follows.  The four
smaller parts can be built-up by the induction hypothesis.  Once they
are done, the line starting from the middle black cube can be
built-up.  After this step the lines starting from $(1,1,1)$ can be
built-up.  We now have three orthogonal facial sections filled.  We
can easily build-up the rest of $[n]^3$: for instance line by line from
2 to $n$ on level 2, and then repeating this level by level from 3 to
$n$.  
\qed \end{proof}

\begin{figure}[ht]
  \begin{center}
    \includegraphics[width=0.75\textwidth]{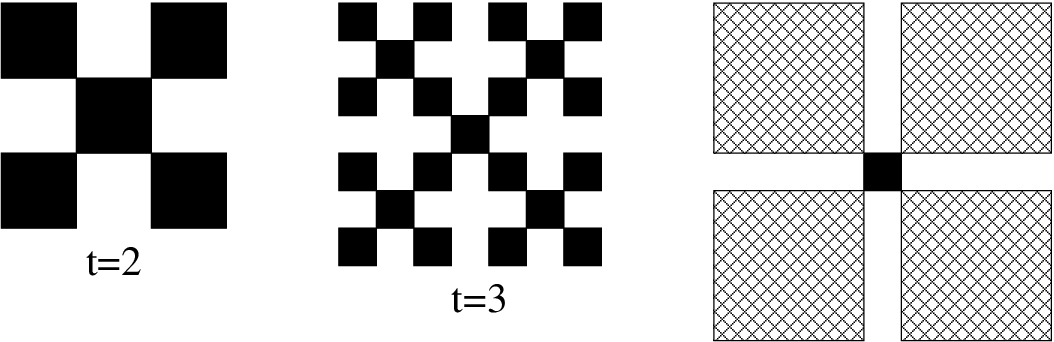}
    \caption{\label{pepita}The recursive idea 
      to create $(n-1)\times (n-1)$ squares}
  \end{center}
\end{figure}

We can look at this result from another point of view.  In the case of
perfect base positions, the $d$ orthogonal projections cover an `area'
of $dn^{d-1}$, which is clearly the maximum possible. At the other
extreme, define $s_d(n)$ to be the minimum possible area obtained from
the $d$ orthogonal projections of a solution to the faces of $[n]^d$.
The following conjecture, posed in \cite{hj2011}, says the above
solution within three planes achieves the minimal value $s_3(n)$.

\begin{conjecture}
$s_3(n)\ge n^2+6n-4$ for all $n$. 
\end{conjecture}

Using the Perimeter Lemma again, we can show that the leading term in 
this conjecture is correct.

\begin{proposition}\label{p:lower}
If $B$ is a solution in $[n]^d$ then each of its $d$
orthogonal projections contains at least $\frac1d n^{d-1}+\frac{d-1}{d} n^{d-2}$ hypercubes of dimension $d-1$.
\end{proposition}

\begin{proof}
Consider a projection of $B$.
Let $x$ be the number of $(d-1)$-dimensional hypercubes in the projection, and
let $p$ be their $(d-2)$-dimensional perimeter.
Now $p\le 2(d-1)x$, with equality if no pair of the $x$ hypercubes 
have a $(d-2)$-dimensional face in common.
Consider how the projection changes as we build-up $[n]^d$ from
the base position.  The only change happens when we add a cube whose
projection was previously empty. It must have $d$ black neighbours
in the projection by the degree condition.  Therefore, the
perimeter decreases by two.  When $[n]^d$ is
completed, the projection is the whole of $[n]^{d-1}$, 
with perimeter $2(d-1)n^{d-2}$.  
Hence $2(d-1)x\ge p=2(d-1)n^{d-2}+2(n^{d-1}-x)$.  Therefore, 
$x\ge\frac1d n^{d-1}+\frac{d-1}{d} n^{d-2}$.
\qed
\end{proof}

In particular, \pref{p:lower} implies that $s_3(n)\ge n^2+2n$.

\section{Enumeration}\label{search}

In this section we give the results of computer enumerations of
solutions for small orders in 3 dimensions, and also show that most
Latin hypercubes (in any dimension) do not correspond to solutions. We
start by considering the 3-dimensional case. By \pref{p:convperfLS}, each
perfect solution corresponds to a Latin square. The number of Latin
squares of order $n$ grows very rapidly as a function of $n$, and is
known only for $n\le11$, see \cite{LS11}. However, \cref{c:LSgreedy}
does at least provide a fast algorithm for testing whether a given
Latin square yields a solution.

A base position $B$ corresponds to a two-colouring of $[n]^3$, where the
vertices in $B$ are black, and the other vertices are white. The
isometry group of the cube has 48 elements \cite{cc}.  Two base positions $B_1$
and $B_2$ are {\it isometric} if and only if there exists an element
of the isometry group of the cube mapping $B_1$ to $B_2$.  Therefore,
any base position may be isometric to at most 48 positions.  If the
position has some symmetry, then we get fewer isometric positions.

Let $S_n$ denote the symmetric group on $n$ letters.  There is a
natural action of the direct product $S_n\times S_n\times S_n$ on the
set of Latin squares of order $n$, which permutes the rows, columns
and symbols of the Latin square. This action is called {\it isotopism}
and its orbits are called {\it isotopism classes}. There is another
operation called {\it conjugacy} in which $S_3$ acts on the Latin
squares by permuting the coordinates of each of the (row, column,
symbol) triples. The closure of an isotopism class under conjugacy is
known as a {\it species} or {\it main class}.

In the present setting, we are thinking of Latin squares as
corresponding to perfect base positions. In this setting, the rows,
columns and symbols of the Latin square correspond to the three
coordinates of cubes in the base position. Therefore, conjugacy of the
Latin square corresponds to an isometry of $[n]^3$ that simply permutes
the $3$ coordinates. As we are interested in the first instance in
counting solutions up to isometry, it follows that it suffices for us
to take one representative of each species and consider all squares in
its isotopism class. We did this for each order $n\le6$, checking each
candidate with the greedy algorithm to test if it was a solution.
Every solution that we found was checked against previously stored
solutions.  If it was not isometric to something already in our
catalogue, then we added it.  In this way, we found all perfect
solutions of order $n\le6$ up to isometry. The number of such
solutions is given in \Tref{T:numsol}.  The Latin squares that produce
the perfect solutions of order 4 are given in \fref{f:LS4}.

\begin{table}[h]
\begin{centering}
\begin{tabular}{|c|cccccc|}
\hline
$n$&1&2&3&4&5&6\\
\hline
\# Latin squares&1&2&12&576&161280&812851200\\
\# perfect solutions&1&2&12&256&2688&148958\\
\# perfect solutions up to isometry&1&1&2&21&72&3697\\
\# all solutions&1&2&116&7134840&&\\
\# all solutions up to isometry&1&1&7&149955&&\\
\hline
\end{tabular}
\caption{\label{T:numsol}Data for small orders}
\end{centering}
\end{table}

\begin{figure}
\[\setlength{\arraycolsep}{1pt}
\begin{array}{ccccccc}
\left[\begin{array}{cccc}
1&2&3&4\\
2&1&4&3\\
3&4&1&2\\
4&3&2&1\\
\end{array}\right]&
\left[\begin{array}{cccc}
1&2&4&3\\
2&1&3&4\\
4&3&1&2\\
3&4&2&1\\
\end{array}\right]&
\left[\begin{array}{cccc}
1&3&4&2\\
3&1&2&4\\
4&2&1&3\\
2&4&3&1\\
\end{array}\right]&
\left[\begin{array}{cccc}
1&4&2&3\\
4&1&3&2\\
2&3&1&4\\
3&2&4&1\\
\end{array}\right]&
\left[\begin{array}{cccc}
1&4&3&2\\
4&1&2&3\\
3&2&1&4\\
2&3&4&1\\
\end{array}\right]&
\left[\begin{array}{cccc}
2&1&4&3\\
1&2&3&4\\
4&3&2&1\\
3&4&1&2\\
\end{array}\right]&
\left[\begin{array}{cccc}
1&4&3&2\\
4&1&2&3\\
2&3&4&1\\
3&2&1&4\\
\end{array}\right]
\\[-1ex]\\
\left[\begin{array}{cccc}
1&2&3&4\\
2&3&4&1\\
3&4&1&2\\
4&1&2&3\\
\end{array}\right]&
\left[\begin{array}{cccc}
1&2&4&3\\
2&4&3&1\\
4&3&1&2\\
3&1&2&4\\
\end{array}\right]&
\left[\begin{array}{cccc}
1&4&3&2\\
4&3&2&1\\
3&2&1&4\\
2&1&4&3\\
\end{array}\right]&
\left[\begin{array}{cccc}
2&1&3&4\\
1&3&4&2\\
3&4&2&1\\
4&2&1&3\\
\end{array}\right]&
\left[\begin{array}{cccc}
1&3&2&4\\
3&4&1&2\\
4&2&3&1\\
2&1&4&3\\
\end{array}\right]&
\left[\begin{array}{cccc}
2&1&4&3\\
1&3&2&4\\
3&4&1&2\\
4&2&3&1\\
\end{array}\right]&
\left[\begin{array}{cccc}
1&3&4&2\\
3&2&1&4\\
4&1&2&3\\
2&4&3&1\\
\end{array}\right]
\\[-1ex]\\
\left[\begin{array}{cccc}
1&3&2&4\\
3&4&1&2\\
2&1&4&3\\
4&2&3&1\\
\end{array}\right]&
\left[\begin{array}{cccc}
1&4&3&2\\
4&2&1&3\\
3&1&2&4\\
2&3&4&1\\
\end{array}\right]&
\left[\begin{array}{cccc}
2&1&4&3\\
1&3&2&4\\
4&2&3&1\\
3&4&1&2\\
\end{array}\right]&
\left[\begin{array}{cccc}
1&2&3&4\\
3&4&2&1\\
4&3&1&2\\
2&1&4&3\\
\end{array}\right]&
\left[\begin{array}{cccc}
1&2&4&3\\
4&3&2&1\\
3&4&1&2\\
2&1&3&4\\
\end{array}\right]&
\left[\begin{array}{cccc}
1&2&3&4\\
2&1&4&3\\
3&4&2&1\\
4&3&1&2\\
\end{array}\right]&
\left[\begin{array}{cccc}
1&2&4&3\\
2&1&3&4\\
4&3&2&1\\
3&4&1&2\\
\end{array}\right]
\end{array}
\]
\caption{\label{f:LS4}The perfect solutions of order $4$.}
\end{figure}

\begin{figure}[ht]
  \begin{center}
    \includegraphics[width=0.75\textwidth]{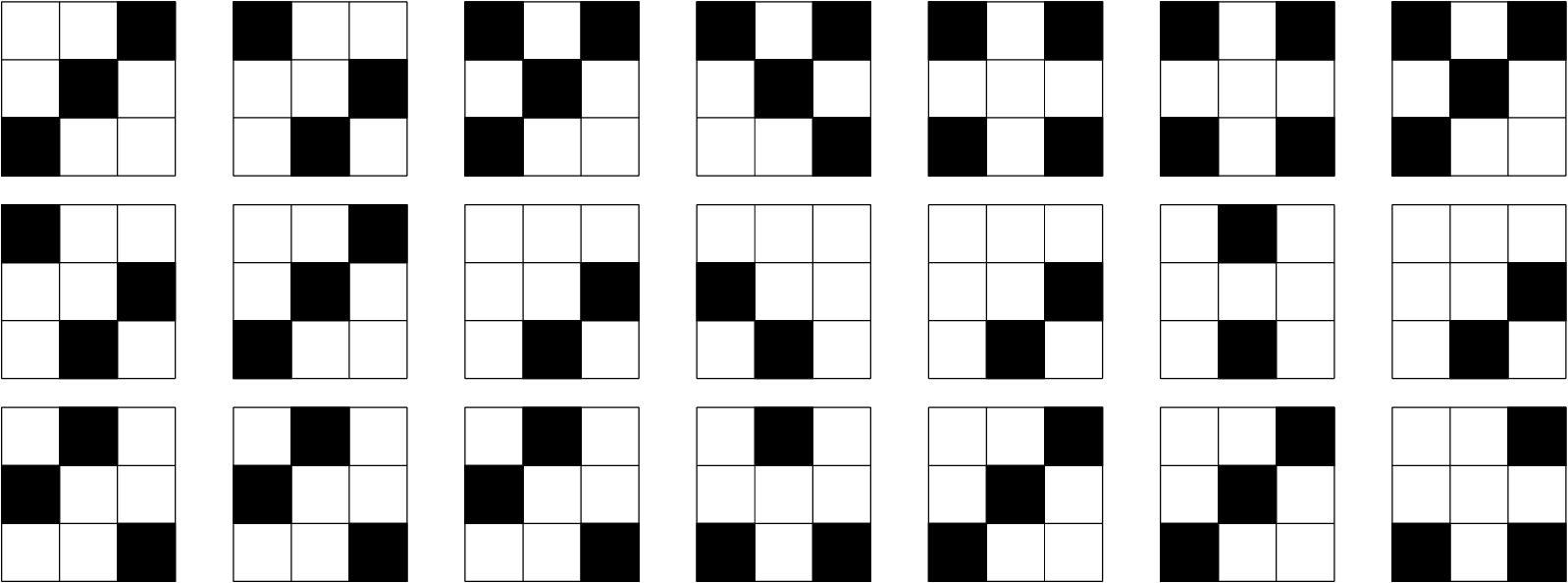}
    \caption{\label{f:all3x3}The seven solutions of order $3$ depicted vertically by level}
  \end{center}  
\end{figure}

In the course of the computations for orders $n\le6$, we observed
that every isotopism class contained at least one solution. We tested
whether this was also true for $n=7$ by a partial enumeration that
stopped examining an isotopism class when it found the first solution. 
We found that there is a unique isotopism class of order 7 that contains
no solutions. A representative of this class is:
\begin{equation*}\label{e:nosol}
\left[
\begin{array}{ccccccc}
1&2&4&5&6&7&3\\
2&1&3&4&7&6&5\\
7&3&2&1&5&4&6\\
3&4&6&2&1&5&7\\
4&7&5&6&3&1&2\\
5&6&7&3&4&2&1\\
6&5&1&7&2&3&4\\
\end{array}
\right].
\end{equation*}
This Latin square is isotopic to all 6 of its conjugates, so its species
consists of a single isotopism class. It has been written in semisymmetric
form, meaning that 3 of its conjugates are actually equal. The only isotopism
that maps it to itself is the identity.

\bigskip

The other enumeration that we did was to look for all solutions of
order $n\le4$, whether perfect or not. This was achieved by adding one
cube at a time, screening for isometry after each of the first five
cubes were added.  Each position with five cubes was extended in all
possible ways to get $n^2$ independent cubes. These were then tested
with the greedy algorithm (as justified by \tref{c:LSgreedy}) to see if
they were solutions. All solutions that were found were screened for
isometry. The number of solutions, up to isometry, is given in
\Tref{T:numsol}.  The $7$ solutions of order 3 are displayed in
\fref{f:all3x3}.  In both \fref{f:LS4} and \fref{f:all3x3}, exactly
one representative of each isometry class is given. Once we had
a catalogue of representatives of each isometry class, we found
the symmetries of these representatives. Applying the
orbit-stabiliser theorem then allowed us to calculate
the number of all solutions and number of perfect solutions,
as given in \Tref{T:numsol}.

It seems plausible on the basis of \Tref{T:numsol} to conjecture that
most Latin squares do not correspond to solutions. We next show that
something stronger is true. In our discussion, use of the word ``random''
will implicitly assume a uniform discrete distribution.  For
properties of random Latin squares, see \cite{CGW08} and the
references therein. To date there does not seem to be much work
on random Latin hypercubes of larger dimensions.

\begin{theorem}\label{t:mostLSfail} 
  Fix a dimension $d\ge2$. Let $\mathcal{R}_n$ denote the base
  position in $[n]^d$ corresponding to a random Latin hypercube of
  order $n$ and dimension $d-1$.  The probability that $\mathcal{R}_n$
  is a solution tends to $0$ as $n\rightarrow\infty$.
\end{theorem}

\begin{proof}
We start with the two-dimensional case.  In \cite{kings}, Shapiro and
Stephens studied a percolation process in which cells with two or more
occupied neighbours became filled. They showed that if the starting
configuration corresponds to a random permutation, then the
probability of completion tends to zero. Since a permutation is just a
Latin hypercube of dimension $1$, and corresponds to a convex
position, our result for $d=2$ follows from \pref{p:uniqmax}.

Next we reduce the case of general $d$ to the two-dimensional case.
Consider the set $T$ of the $n^2$ cubes in $[n]^d$ that have their
first $d-2$ coordinates all equal to $n$. The cubes in $T$ that are
black in $\mathcal{R}_n$ correspond to a permutation $\pi$, since each
line contains exactly one black cube. Moreover, $\pi$ is random since,
by symmetry, every rearrangement of the sections of $\mathcal{R}_n$
that have last coordinate fixed was equally likely to occur.  Now
since cubes in $T$ have only $d-2$ neighbours in $[n]^d$ outside
$T$, it will be impossible to build-up $T$ from the black cubes in
$\mathcal{R}_n$ unless $\pi$ percolates in the Shapiro and Stevens
model. We know that event has probability tending to $0$ as
$n\rightarrow\infty$.  \qed \end{proof}

\section{The cyclic base position in three dimensions}\label{s:cyclic}

The graph $[n]^3$ consists of $n$ levels, each level being an $n\times
n$ square of unit cubes.  Each vertex is identified naturally with a
triple $(x,y,z)$.  Here the last coordinate indexes the level, and
the first two coordinates are the row and column indices.  
The rows are $1,2,\dots,n$ from top to bottom and the columns are $1,2,\dots,n$ 
from left to right as in a matrix, the levels
are $1,2,\dots, n$ from bottom to top.  
We describe a specific solution, the {\it cyclic base
  position}, that can be reached after dismantling $[n]^3$.  The set of
black vertices is $\{(i,j,k):i+j-k\equiv 1 \mod n\}$.  The top level
contains diagonal vertices.  Each consecutive level is a cyclic shift
of the previous level, see Figure~\ref{cyc5} for order 5.

\begin{figure}[ht]
  \begin{center}
    \includegraphics[width=0.75\textwidth]{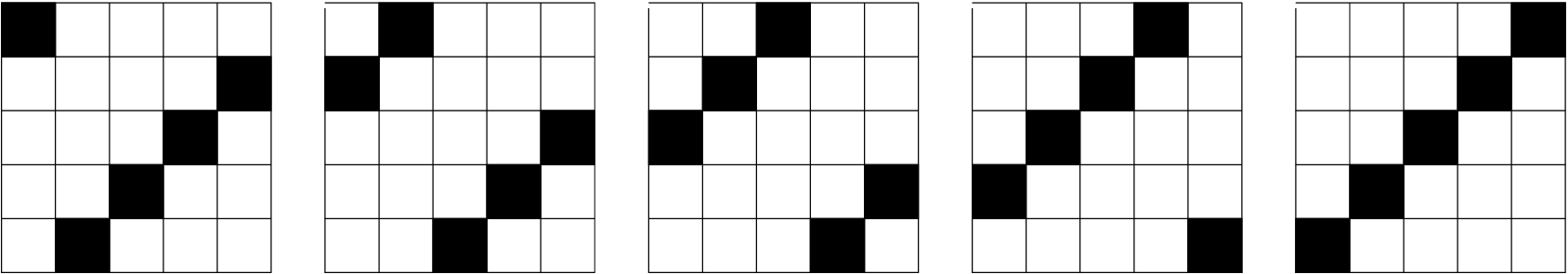}
    \caption{\label{cyc5}The cyclic base position of order $5$}
  \end{center}
\end{figure}

We define a specific subroutine, which is used several times.  Let $S$
be a facial section of $[s]^3$, that corresponds naturally to an
$s\times s$ square.  Let $D$ be a diagonal of black vertices
$(1,m),(2,m-1),\dots,(m,1)$, where $m\le s$.  Consider the upper-left
triangle $T$ consisting of the vertices with $i+j\le m$.  Now the
diagonals parallel to $D$ can be removed starting from size $1$ to
size $m-1$.  This process is the {\it diagonal peeling} of $T$ or a
diagonal peeling of size $m-1$ with corner $(1,1)$.  Any rotation or
reflection of it is also a diagonal peeling.

Let $UTB(m)=\{(i,j,k):i+j-k=m+1-n\}$ be the {\it upper triangular base} of order $m$, where $1\le m\le n$.
Similarly, let $LTB(m)=\{(i,j,k):i+j-k=2n-m\}$ be the {\it lower triangular base} or order $m$, where $1\le m\le n-1$.
The cyclic base position in $[n]^3$ is the union of $UTB(n)$ and $LTB(n-1)$.
In general $TB(m)$, a triangular base of order $m$ is a copy of $UTB(m)$ embedded and possibly rotated and reflected in $[n]^3$, where $m\le n$.
Clearly $LTB(n-1)$ is a $TB(n-1)$. 

The union $\cup_{i=1}^{m}UTB(i)$ forms a {\it heap of oranges} of order
$m$, $HO(m)$ for short.  For any $m$, where $1\le m\le n-1$, the vertices of
$UTB(m)$ are independent, and any vertex of $UTB(m)$ has precisely
three neighbours in $UTB(m+1)$.  Therefore, $HO(n)$ can be dismantled
to $UTB(n)$.

The vertices of $[n]^3$ with precisely 3 neighbours in $HO(n)$ form a $TB(n-2)$.
They are on the other side of $UTB(n)$.
Add these vertices to $HO(n)$.
Repeating this process, we can pack a $TB(n-4),TB(n-6),\dots$ on $HO(n)$. 
This is the {\it nested heap of oranges} construction, $NHO(n)$ for short.
If $n$ is even, then $NHO(n)$ is $HO(n)\cup TB(n-2)\cup \cdots\cup TB(2)$.
If $n$ is odd, then $NHO(n)$ is $NHO(n)\cup TB(n-2)\cup \cdots\cup TB(1)$.
Observe that $NHO(n)$ can be dismantled to $UTB(n)$,
and a rotated copy of $NHO(n-1)$ can be dismantled to $LTB(n-1)$.

\begin{theorem} \label{cyc}
For each positive integer $n$, where $n\ge 2$, the $n\times n\times n$
cube can be dismantled to the disjoint union of $NHO(n)$ and a rotated
copy of $NHO(n-1)$.
 \end{theorem}

\begin{proof}
Consider the three original corners of $[n]^3$ outside of
$NHO(n)$ and $NHO(n-1)$: $(1,n,1)$, $(n,1,1)$, $(n,n,n)$.  Execute
diagonal peelings of size $n-1$ with the above three corners to remove
vertices of form $(1,y,z)$, $(x,1,z)$, $(x,y,n)$.  We remove 
$3T_{n-1}$ cubes in this step, where $T_m=m(m-1)/2$ is the triangular
number.  In particular, we remove $(1,2,1)$, $(2,1,1)$, $(n,1,n-1)$,
$(n,2,n)$, $(1,n,n-1)$, $(2,n,n)$.  Therefore, the following vertices
now have three neighbours: $(2,2,1)$, $(n,2,n-1)$, $(2,n,n-1)$.
Execute diagonal peelings of size $n-2$ with the above corners to
remove vertices of form $(x,y,1)$, $(n,y,z)$, $(x,n,z)$.  In these two
steps, we removed all vertices outside of $NHO(n)$ and $NHO(n-1)$ with
any coordinate equal to $1$ or $n$.  In some sense, we peeled of one
hull from the cube.

{}From now on, we repeat these two steps from corners, each of which is a diagonal step away from one of the previous six.
The series of corners are: $(1+a,n-a,1+a)$,
$(n-a,1+a,1+a)$, $(n-a,n-a,n-a)$, $(2+b,2+b,1+b)$, $(n-b,2+b,n-1-b)$,
$(2+b,n-b,n-1-b)$, where $0\le a\le \lfloor n/2\rfloor-1$ and $0\le b\le
\lceil n/2\rceil-2$.  So the second iteration would use diagonal
peelings from the corners $(2,n-1,2)$, $(n-1,2,2)$, $(n-1,n-1,n-1)$,
and $(3,3,2)$, $(n-1,3,n-2)$, $(3,n-1,n-2)$ to remove vertices with
any coordinate equal to $2$ or $n-1$.  Each new peeling is made
possible by the preceding peeling, which removed the cubes that were
neighbours on the ``outside'' of the ones we want to remove next.

We can perform $n-1$ steps of this algorithm removing $(n-1)n(n+1)/2$
cubes in total, 3 times the $(n-1)^{\rm st}$ tetrahedral number $H_{n-1}$.
Adding these to the number of cubes in $NHO(n)$ and $NHO(n-1)$ gives
us $n^3$, certifying our algorithm.  Indeed,
\begin{align*}
|NHO(n)|+|NHO(n-1)|
&=H_{n}+T_{n-2}+T_{n-4}+\cdots\\
&\quad+H_{n-1}+T_{n-3}+T_{n-5}+\cdots\\
&=H_{n}+H_{n-1}+H_{n-2}.
\end{align*}
Therefore, we need to check $H_{n}+4H_{n-1}+H_{n-2}=n^3$, which is true.
\qed \end{proof}

\begin{corollary}\label{cy:cyclic}
For each positive integer $n$, where $n\ge 2$, the $n\times n\times n$
cube can be dismantled to the cyclic base position.
\end{corollary}

If we are looking for more solutions, it is natural to consider
variations of the cyclic base position.  One option is to permute the
levels, which corresponds to permuting the symbols in the
corresponding latin square.  Executing a computer search, we found
that of the $n!$ permutations of the symbols, the number that produced
a solution is as follows.
\[
\begin{tabular}{|c|cccccccc|}
\hline
$n$&3&4&5&6&7&8&9&10\\
\hline
\#solutions&6&16&40&96&200&352&552&800\\
\hline
\end{tabular}
\]
We will not consider general permutations further, but rather work
towards showing that cyclic permutations do produce solutions.

Starting from a corner of $[n]^3$ we can remove a line of cubes of
length $k$, where $1\le k\le n-1$.  As a consequence, we can build-up
a missing line of cubes on the edge of a cube.  We use this
observation in a more general context, when the line of cubes is
somewhere inside $[n]^3$, but a dismantling generated an equivalent
situation.

A {\it staircase} of size $m$ and depth $t$ is the union of lines of
length $t$ starting from a set of vertices that form a diagonal
peeling of size $m$.

\begin{lemma} \label{stair}
 We can remove a staircase of size $m$ and depth $t$ from $[n]^3$, where
 $1\le m<n$ and $1\le t<n$.
\end{lemma}

\begin{proof}
We can either say that this is the union of $t$ diagonal
peelings of size $m$, or refer to the repeated removal of lines in
diagonal fashion.  
\qed \end{proof}

The cyclic base position is the intersection of $[n]^3$ with two planes.
When we permute the levels cyclically, we get the set
$\{(i,j,k):i+j-k\equiv 1-s \mod n\}$, where $1\le s\le n-1$.  It is
the intersection of $[n]^3$ with three planes according to whether
$i+j-k$ takes the value $1-s$, $n+1-s$ or $2n+1-s$.  These are two
triangular and one hexagonal region.  Therefore, let
$HB_n^s=\{(i,j,k):i+j-k=n+1-s\}$ be the {\it hexagonal board} of size
$s$ in $[n]^3$, where $1\le s\le n-1$.
Let $HB(m+1-s,s)$ denote any rotated or reflected copy of $HB_m^s$
embedded in $[n]^3$, where $m\le n$.  In this way $HB(1,1)$ is a single
cube, $HB(2,1)$ corresponds to three cubes, $HB(2,2)$ consists of
seven cubes etc.  As for the triangular boards, we can put together
hexagonal boards in a nested fashion.  Let $m$ be a positive integer
and $P_m=1$ if $m$ is odd, $P_m=2$ if $m$ is even.  Let $DNH(n,m)$ be
the {\it doubly nested hexagon} of size $m$, which is the union of
$HB(P_m,1),\dots, HB(m,1),\dots, HB(m,n-m+1),\dots,$ $HB(1,n-m+1),\dots, HB(1,P_m)$.  
As remarked before, the nested parts can be dismantled.
That is, $DNH(n,m)$ can be dismantled to $HB(m)$.

We now prove the following generalisation of \tref{cyc}.

\begin{theorem}
 Any cyclic permutation of the levels of the cyclic base position
 gives a solution.
\end{theorem}

\begin{proof} 
Let as assume that we shifted the levels by $s$.  That is,
level $n$ of the cyclic base position becomes level $s$, and in
general level $i$ becomes level $i+s$ $\mod n$.  We denote this object
by $CC(n,s)$.  It is the union of three connected pieces: $UTB(n-s)$,
$HB_n^s$, $LTB(s-1)$, see picture 1, in \fref{hexbuild}.  As in
the preparation for Theorem~\ref{cyc} we can add cubes to these three
pieces to get $NHO(n-s)+DNH(n,s)+NHO(s-1)=M$, the third picture in
\fref{hexbuild}.

\begin{figure}[ht]
  \begin{center}
    \includegraphics[scale=0.2]{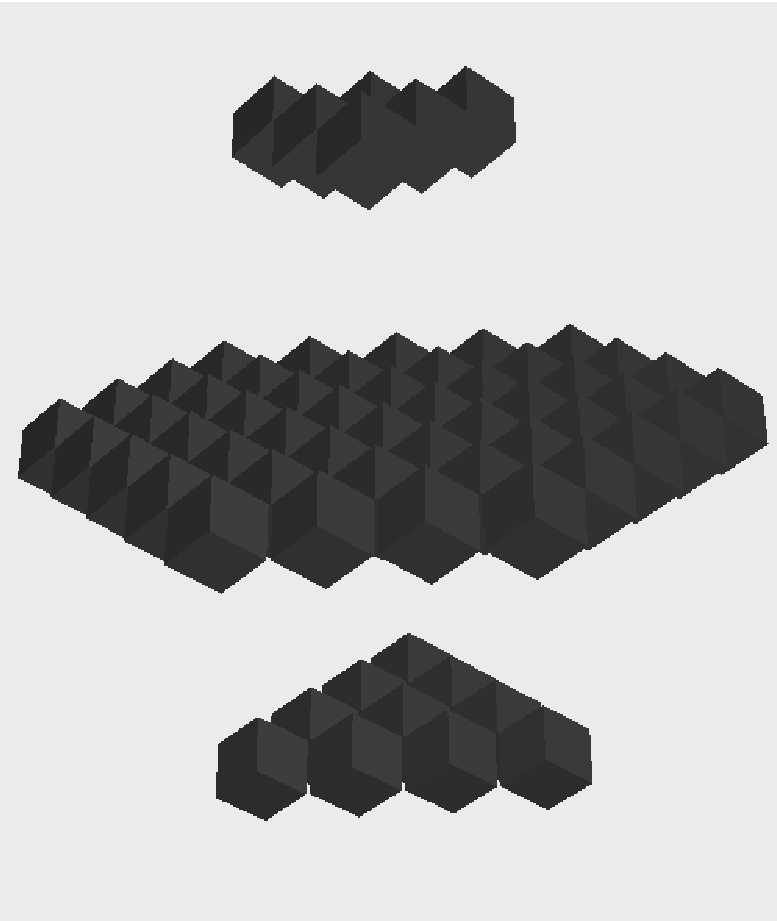}\includegraphics[scale=0.2]{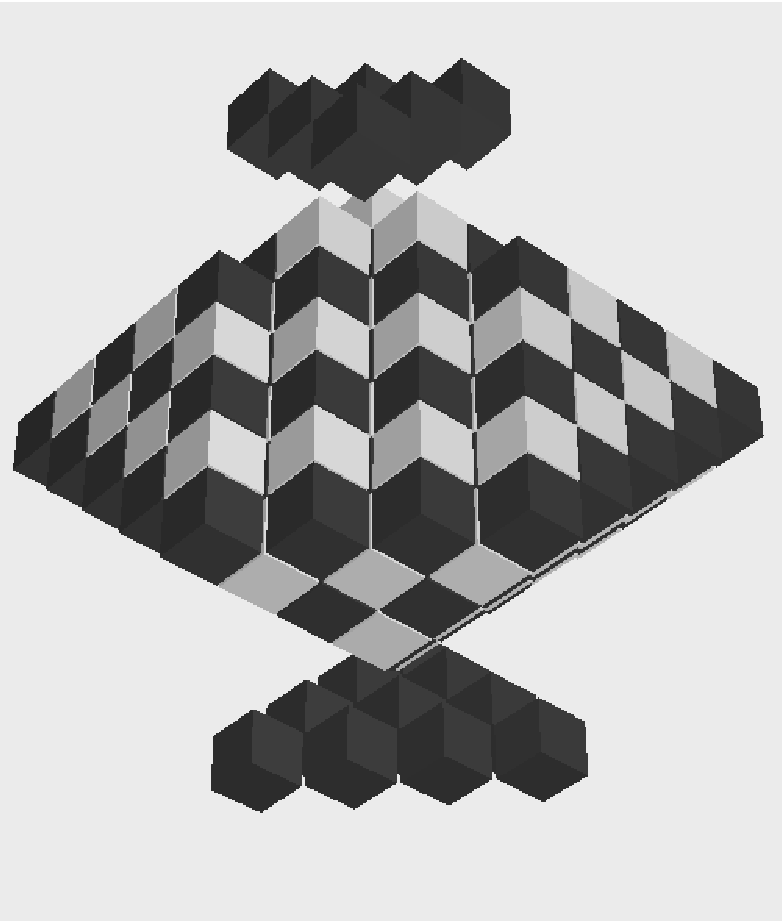}\includegraphics[scale=0.2]{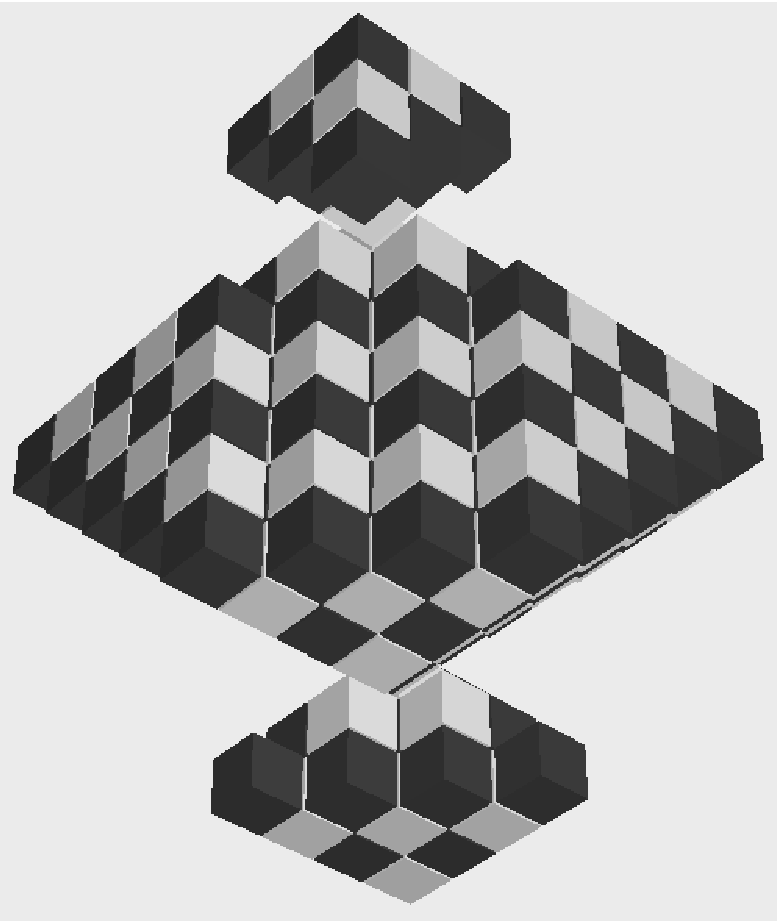}\includegraphics[scale=0.2]{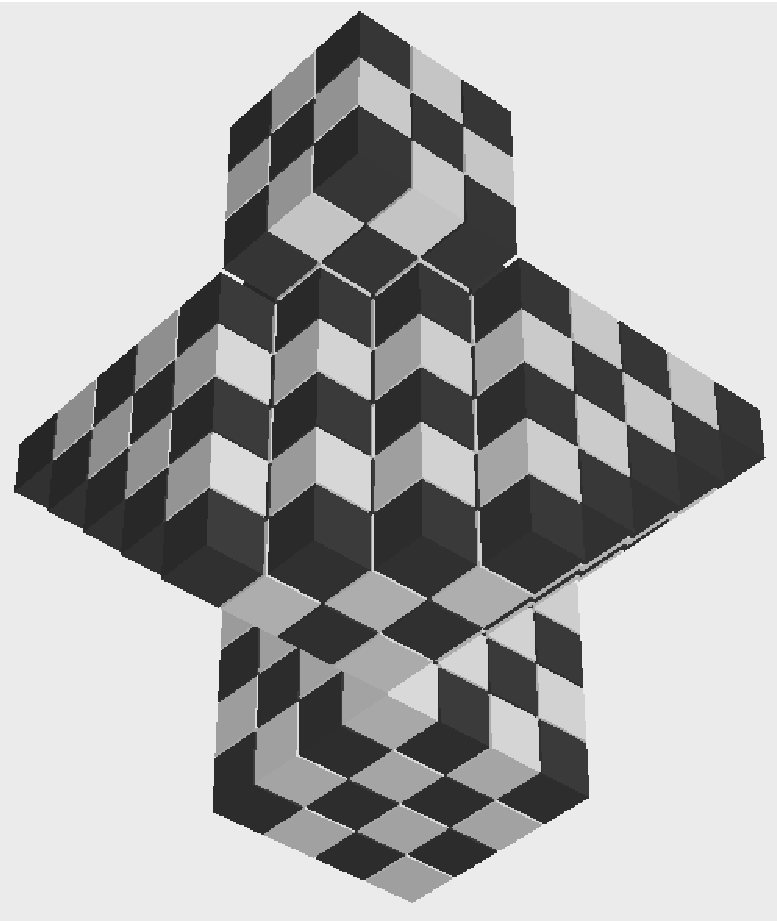}
    \caption{\label{hexbuild}Build-up from a hexagonal board via double nesting and completion of corners}
  \end{center}
\end{figure}

The key observation is that the intersection of $M$ and a $[n-s]^3$
placed in $(n,n,n)$ is a copy of $NHO(n-s)$ and $NHO(n-s-1)$.
Therefore, we can apply \tref{cyc}.  Similarly for the intersection of
$M$ and a $[s-1]^3$ placed in $(1,1,1)$.  Therefore, we can locally
build-up $[n-s]^3$ and $[s-1]^3$, see the fourth picture in
\fref{hexbuild}.  Let $F$ be the union of these two cubes with
$DNH(n,s)$.

Using the above observations, we first use \lref{stair} to
dismantle $[n]^3$ to $F$.  Secondly, we use Theorem~\ref{cyc} to
dismantle $F$ to $M$.  Finally, we dismantle $M$ to $CC(n,s)$.  
\qed \end{proof}

\section{Cuboids}

Some of our results immediately generalise to cuboids.  Let $C(k,l,m)$
be the cuboid consisting of $klm$ unit cubes.  \pref{p:opti}
gives that $\lceil \frac{kl+lm+mk}{3}\rceil$ is the minimum number of
unit cubes after any dismantling. This immediately implies that
solutions of independent vertices only exist if at least two of
$k,l,m$ are divisible by 3 or they all have the same residue modulo 3.
On the other hand, non-independent optimal positions might be more
frequent or easier to find.

\lref{convex} and Theorems~\ref{t:notcontain}, \ref{anytwo} generalise
directly.  Unfortunately, even if the cyclic base position has some
analogues, they give no solution in $C(k,l,m)$.  Therefore, it is of
great interest to find any solution in those cases.

\begin{conjecture}
There is a solution for $C(2,2,k)$ if and only if $k=5$. \\
There is a solution for $C(2,3,k)$ if and only if $k=6$.\\ 
There is a solution for $C(3,3,k)$ if and only if $k$ is odd or $k=4$.
\end{conjecture}

The following picture shows the solutions in the above cases.

\begin{figure}[ht]
  \begin{center}
    \includegraphics[width=0.75\textwidth]{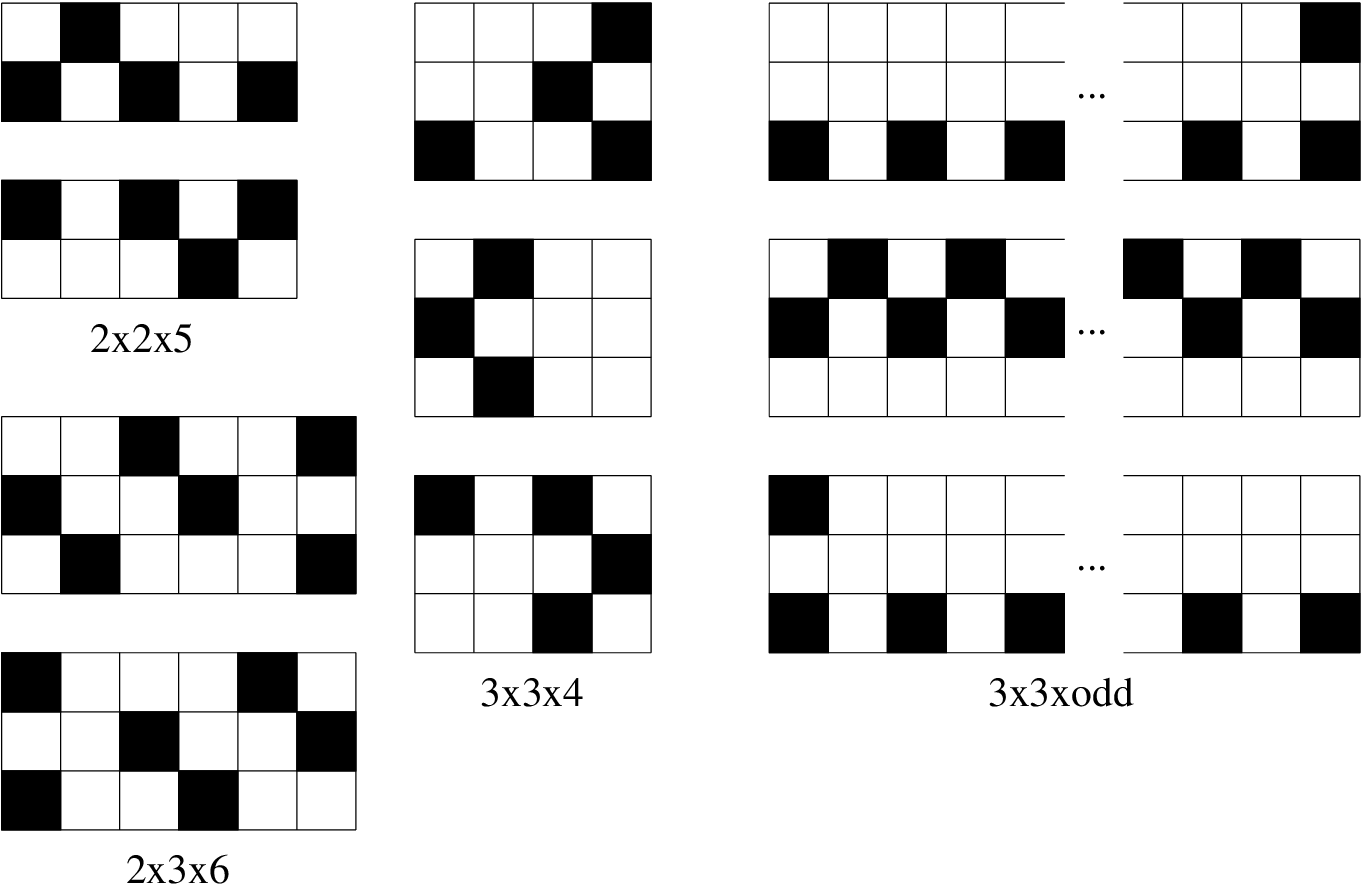}
    \caption{\label{attach}Solutions for small cuboids}
  \end{center}
\end{figure}

\section{Open problems}

In this final section we record some ideas that might be interesting to
pursue as extensions of the current paper.

The first class of open questions involves trying to make consecutive
moves from neighbouring vertices, so that the sequence of moves form a
path.  It is not possible to always obey this constraint; the first
move removes a corner of the hypercube, then you would have to work
along one edge of the hypercube, but there is no way to progress to
anywhere else. So one question is to find the fewest number of paths
that can be combined to complete a dismantling. This would be one
measure of the efficiency of the dismantling, imagining that it is
perhaps being done by a crane that takes time to move around.  An
alternative to minimising the number of paths would be to try to
minimise the total distance travelled by the crane when dismantling.
Yet another variant would be to consider cubes to be neighbours if
they meet at even just a point, then to ask for the minimum number of
paths in a dismantling. It is not clear in this variant whether it is
possible to dismantle in a single path.

The second idea is to consider the game version of dismantling. Two players
would alternate in removing vertices of degree $d$, starting from
$[n]^d$.  The one to make the last valid move wins (or, in a different
variant, the last player to move loses). Which player has a winning
strategy?  Since $[n]^d$ is a bipartite graph, it would also be possible
to two colour $[n]^d$ and allocate one colour to each player. Then you
could insist that each player can only remove vertices of their own
colour.

As $[n]^d$ is bipartite it would also be interesting to ask what are the
solutions that lie entirely in one colour class.  If we imagine the
unit cubes are coloured with two colours like in a checker board, then
we are seeking monochromatic solutions. The cyclic base position is
monochromatic if and only if $n$ is even.  The construction within
three facial sections, see Figure~\ref{pepita}, is another example.

To tie in with the extensive literature on bootstrap percolation, it
would be worth considering the minimal positions that can be reached
by dismantling from $[n]^d$. By analogy, these could be called {\it
  minimal percolating sets}. What is the maximum size of such a
minimal percolating set?

On the basis of evidence such as \Tref{T:numsol} it seems reasonable
to conjecture that most solutions in $3$-dimensions are not
perfect. However, we have not even shown the existence of an imperfect
solution for all orders. In \sref{s:cyclic}, we did show that there is
a perfect solution for all orders, and in \sref{s:geom} we also showed
an infinite family of imperfect solutions.

Finally, we note that in \tref{t:mostLSfail} we showed that almost all
Latin squares do not correspond to solutions. However, it is an open
question whether most isotopism classes of Latin squares contain an example
that corresponds to a solution.

\begin{acknowledgements}

We would like to thank Viktor V\'\i gh for bringing this problem into
our attention as well as reference \cite{nandi}.  The dismantling of
$[5]^3$ with gravity is mentioned in a solution of a problem posed in a
mathematical competition for elementary school pupils \cite{nandi}.
The gravity version distinguishes a bottom level and allows
removal of cubes from the top of a position only.  That is, if the
removal of a cube would cause another cube to fall, then the move is
forbidden.  Our version with zero gravity \cite{hj2011} captures the
truly combinatorial flavour of the problem.  The first author is
grateful for M\'ark Korondi and Vir\'ag Varga for fruitful
discussions, computational work and visualisation.

\end{acknowledgements}

\end{document}